\documentclass[a4paper,11pt]{article}
\usepackage[sc]{mathpazo}
\usepackage[T1]{fontenc}
\usepackage{amsthm,amsmath,amssymb,amsfonts,hyperref,a4wide}
\usepackage[toc,page]{appendix}
\usepackage[usenames,dvipsnames]{xcolor} %% It is important to load xcolor before tikz
\usepackage{tikz}
\makeatletter
\tikzset{my loop/.style =  {to path={
  \pgfextra{}
  [looseness=12,min distance=6mm]
  \tikz@to@curve@path},font=\sffamily\small
  }}
\makeatletter
\usetikzlibrary{matrix,arrows,backgrounds,positioning,fit,shapes}

\newtheorem{theorem}{Theorem}[section]
\newtheorem{lemma}[theorem]{Lemma}
\newtheorem{corollary}[theorem]{Corollary}

\newtheorem{proposition}[theorem]{Proposition}
\newtheorem{prop}[theorem]{Proposition}
\theoremstyle{definition}

\theoremstyle{remark}

\newtheorem{question}{Question}

\newcommand{\change}[1]{{\color{black}#1}}

\newcommand*{\R}{\mathbb{R}}

\newcommand*{\N}{\mathbb{N}}
\newcommand*{\C}{\mathbb{C}}

\newcommand*{\eps}{\varepsilon}

\title{On a conjecture of Sokal concerning roots of the independence polynomial}
\author{Han Peters \thanks{Korteweg de Vries Institute for Mathematics, University of Amsterdam. Email: \texttt{hanpeters77@gmail.com.}}\and Guus Regts\thanks{Korteweg de Vries Institute for Mathematics, University of Amsterdam. Email: \texttt{guusregts@gmail.com}. Supported by a personal NWO Veni grant.}}
\begin{document}
\maketitle
\abstract{A conjecture of Sokal \cite{S1} regarding the domain of non-vanishing for independence polynomials of graphs, states that given any natural number $\Delta \ge 3$, there exists a neighborhood in $\mathbb C$ of the interval $[0, \frac{(\Delta-1)^{\Delta-1}}{(\Delta-2)^{\Delta}})$ on which the independence polynomial of any graph with maximum degree at most $\Delta$ does not vanish. We show here that Sokal's Conjecture holds, as well as a multivariate version, and prove optimality for the domain of non-vanishing. An important step is to translate the setting to the language of complex dynamical systems.

\begin{footnotesize}
\end{footnotesize}
}

\section{Introduction}
For a graph $G=(V,E)$ and $\lambda=(\lambda_v)_{v\in V}\in \C^V$, the \emph{multivariate independence polynomial}, is defined as
\[
Z_G(\lambda):=\sum_{\substack{I\subseteq V\\\text{independent}}} \prod_{v\in I}\lambda_v.
\]
We recall that a set $I\subseteq V$ is called \emph{independent} if it does not span any edges of $G$.
The \emph{univariate independence polynomial}, which we also denote by $Z_G(\lambda)$, is obtained from the multivariate independence polynomial by plugging in $\lambda_v=\lambda$ for all $v\in V$.

In statistical physics the univariate independence polynomial is known as the partition function of the hardcore model. When $\lambda=1$, $Z_G(\lambda)$ equals the number of independent sets in the graph $G$.

Motivated by applications in statistical physics Sokal \cite[Question 2.4]{S1} asked about domains of the complex plane where the independence polynomial does not vanish.
Just below Question 2.4 in \cite{S1}, Sokal conjectures: \emph{``there is a complex domain $D_\Delta$ containing at least the interval $0\leq \lambda<1/(\Delta-1)$ of the real axis --- and possibly even the interval $0\leq \lambda<\lambda_\Delta:=\frac{(\Delta-1)^{\Delta-1}}{(\Delta-2)^\Delta}$ --- on which $Z_G(\lambda)$ does not vanish for all graphs of maximum degree at most $\Delta$"}.

In this paper we confirm the strong form of his conjecture for the univariate independence polynomial.
In Section~\ref{sec:proof} we will prove the following result:
\begin{theorem}\label{thm:main}
Let $\Delta \in \mathbb N $ with $\Delta \geq 3$. Then there exists a complex domain $D_\Delta$ containing the interval $0\leq \lambda<\lambda_\Delta$ such that for any graph $G=(V,E)$ of maximum degree at most $\Delta$ and any $\lambda \in D_\Delta$, we have that $Z_G( \lambda)\neq 0$.
\end{theorem}
If we allow ourselves an epsilon bit of room, then the same result also holds for multivariate independence polynomial. This is the contents of Theorem \ref{thm:precise} in Section~\ref{sec:proof}. We show in Appendix~\ref{app:complex dynamics} that the literal statement of Theorem~\ref{thm:main} does not hold in the multivariate setting.

It follows from nontrivial results in complex dynamical systems that the bound in Theorem~\ref{thm:main} is in fact optimal, in light of the following:
\begin{proposition}\label{prop:tight}
Let $\Delta \in \mathbb N$ with $\Delta \geq 3$. Then there exist $\lambda \in \mathbb C$ arbitrarily close to $\lambda_\Delta$ for which there exists a graph $G$ of maximum degree $\Delta$ with $Z_G(\lambda)=0$.
\end{proposition}
This result is a direct consequence of Proposition~\ref{prop:regular trees} in Subsection~\ref{sec:regular trees}.
We discuss the underlying results from the theory of complex dynamical systems in Appendix~\ref{app:complex dynamics}.

Other results for the nonvanishing of the independence polynomial include a result of Shearer \cite{Sh85} that says that for any graph $G=(V,E)$ of maximum degree at most $\Delta$ and any $\lambda$ such that for each $v\in V$, $|\lambda_v|\leq \frac{(\Delta-1)^{\Delta-1}}{\Delta^\Delta}$ one has $Z_G(\lambda)\neq 0$. See \cite{SS05} for a slight improvement and extensions.
Moreover, Chudnovsky and Seymour \cite{CS7} proved that the univariate independence polynomial of a claw-free graph (a graph $G$ is called \emph{claw-free} if it does not contain four vertices that induce a tree with three leaves), has all its roots on the negative real axis.

\subsubsection*{Motivation}
Another motivation for Theorem~\ref{thm:main} comes from the design of efficient approximation algorithms for (combinatorial) partition functions.
In \cite{W6} Weitz showed that there is a (deterministic) fully polynomial time approximation algorithm (FPTAS) for computing $Z_G(\lambda)$ for any $0\leq \lambda<\lambda_c(\Delta)$ for any graph of maximum degree at most $\Delta$. His method is often called the \emph{correlation decay method} and has subsequently been used and modified to design many other FPTAS's for several other types of partition functions; see e.g. \cite{BGKNT7,GK12,LY13,LL15}.
More recently, Barvinok initiated a line of research that led to quasi-polynomial time approximation algorithms for several types of partition functions and graph polynomials; see e.g. \cite{B14a,B14b,BS14a,BS14b,R15} and Barvinok's recent book \cite{B17}.
This approach is based on Taylor approximations of the log of the partition function/graph polynomial, and allows to give good approximations in regions of the complex plane where the partition function/polynomial does not vanish.
In his recent book \cite{B17}, Barvinok refers to this approach as the \emph{interpolation method}.
Patel and the second author \cite{PR16} recently showed that the interpolation method in fact yields polynomial time approximation algorithms for these partition functions/graph polynomials when restricted to bounded degree graphs.

In combination with the results in Section 4.2 from \cite{PR16}, Theorem~\ref{thm:main} immediately implies that the interpolation methods yields a polynomial time approximation algorithm for computing the independence polynomial at any fixed $0\leq \lambda<\lambda_\Delta$ on graphs of maximum degree at most $\Delta$, thereby matching Weitz's result.
In particular, Theorem~\ref{thm:main} gives evidence for the usefulness of the interpolation method.

\subsection*{Preliminaries}
We collect some preliminaries and notational conventions here.
Graphs may be assumed to be simple, as vertices with loops attached to them can be removed from the graph and parallel edges can be replaced by single edges without affecting the independence polynomial.
Let $G=(V,E)$ be a graph. For a subset $U\subseteq V$ we denote the graph induced by $U$ by $G[U]$.
For $U\subset V$ we denote the graph induced by $V\setminus U$ by $G\setminus U$; in case $U=\{u\}$ we just write $G-u$.
For a vertex $v\in V$ we denote by $N[v]:=\{u\in V\mid \{u,v\}\in E\}\cup \{v\}$ the \emph{closed neighborhood} of $v$.
The \emph{maximum degree} of $G$ is the maximum number of neighbors of a vertex over all vertices of $G$. This is denoted by $\Delta(G)$.

For $\Delta\in \N$ and $k\in \N$ we denote by $T_{\Delta,k}$ the rooted tree, recursively defined as follows: for $k=0$, $T_{\Delta,0}$ consists of a single vertex; for $k>0$, $T_{\Delta,k}$ consists of the root vertex, which is connected to the $\Delta-1$ root vertices of $\Delta-1$ disjoint copies of $T_{\Delta,k-1}$.
We will sometimes, abusing terminology, refer to the $T_{\Delta,k}$ as \emph{regular trees}.
Note that the maximum degree of $T_{\Delta,k}$ equals $\Delta$ whenever $k\geq 2$ and equals $\Delta-1$ when $k=1$.
\\

\noindent{\bf Organization}
The remainder of this paper is organised as follows. In the next section we translate the setting to the language of complex dynamical systems and we prove another non-vanishing result for the multivariate independence polynomial, cf. Theorem~\ref{thm:main1}.
Section~\ref{sec:coordinate} contains technical, yet elementary, derivations needed for the proof of our main result, which is given in Section~\ref{sec:proof}. We conclude with some questions in Section~\ref{sec:conclusion}.
In the appendix we discuss results from complex dynamical systems theory needed to prove Proposition~\ref{prop:tight}.

\section{Setup}
We will introduce our setup in this section.

%We collect some definitions here and state and prove a lemma that will be very convenient to prove our main result.
Let us fix a graph $G=(V,E)$, $\lambda=(\lambda_v)_{v\in V}\in \C^V$ and a vertex $v_0\in V$.
The fundamental recurrence relation for the independence polynomial is
\begin{equation}
Z_G(\lambda)=\lambda_{v_0} Z_{G\setminus N[v_0]}(\lambda)+Z_{G-v_0}(\lambda). \label{eq:fund recurrence}
\end{equation}
Let us define, assuming $Z_{G-v_0}(\lambda)\neq 0$,
\begin{equation}
R_{G,v_0}=\frac{\lambda_{v_0} Z_{G\setminus N[v_0]}(\lambda)}{Z_{G-v_0}(\lambda)}.\label{eq:def R}
\end{equation}
\change{In the case that $\lambda_v>0$ for all $v\in V$} \eqref{eq:def R} is always defined.
This definition is inspired by Weitz \cite{W6}.
We note that by \eqref{eq:fund recurrence},
\begin{equation}\label{eq:no zero not -1}
R_{G,v_0}\neq -1 \text{ if and only if } Z_G(\lambda)\neq 0.
\end{equation}
So for our purposes it suffices to look at the ratio $R_{G,v_0}$.

\subsection{Regular trees}\label{sec:regular trees}
We now consider the univariate independence polynomial for the trees $T_{\Delta,k}$.
Let $v_k$ denote the root vertex of $T_{\Delta,k}$.
Then for $k>0$, $T_{\Delta,k}-v_k$ is equal to the disjoint union of $\Delta-1$ copies of $T_{\Delta,k-1}$.
Additionally, for $k>1$, $T_{\Delta,k}\setminus N[v_k]$ is equal to the disjoint union of $\Delta-1$ copies of $T_{\Delta,k-1}-v_{k-1}$.
Using this we note that for $k>2$ \eqref{eq:def R} takes the following form:
\begin{align}\label{eq:rec tree}
R_{T_{\Delta,k},v_k}&=\lambda\left(\frac{Z_{T_{\Delta,k-1}-v_{k-1}}}{Z_{T_{\Delta,k-1}}}\right)^{\Delta-1}
=\lambda\left(\frac{Z_{T_{\Delta,k-1}-v_{k-1}}}{\lambda Z_{T_{\Delta,k-1}\setminus N[v_{k-1}]}+Z_{T_{\Delta,k-1}-v_{k-1}}}\right)^{\Delta-1}\nonumber
\\
&=\frac{\lambda}{(1+R_{T_{\Delta,k-1},v_{k-1}})^{\Delta-1}}.
\end{align}

We denote the extended complex plane $\C\cup \{\infty\}$ by $\widehat \C$.
Define for $\lambda\in \C$ and $d\in \N$, $f_{d,\lambda}:\widehat \C \to \widehat \C$ by
\[
f_{d,\lambda}(x)=\frac{\lambda}{(1+x)^{d}}.
\]
So \eqref{eq:rec tree}  gives that $R_{T_{\Delta,k},v_k}=f_{\Delta-1,\lambda}(R_{T_{\Delta,k-1},v_{k-1}})$.
Noting that $R_{T_{\Delta,v_0}}=\lambda$, we observe that $R_{T_{\Delta,k},v_0}=f_{\Delta-1,\lambda}^{\circ k}(\lambda)$.
\change{So to understand under which conditions $R_{T_{\Delta,k},v_k}$ equals $-1$ or not, it suffices to look at the orbits of $f_{\Delta-1,\lambda}$ with starting point $\lambda$, or equivalently with starting point $-1$.}
%So to understand the values of $R_{T_{\Delta,k},v_k}$, it suffices to understand the dynamics of the map $f_{\Delta-1,\lambda}$.

A somewhat similar relation between graphs and the iteration of rational maps was explored by Bleher, Roeder and Lyubich in \cite{BLR2010} and \cite{BLR2011}. While here one iteration of $f_{\Delta-1,\lambda}$ corresponds to adding an additional level to a tree, there one iteration corresponded to adding an additional refinement to a hierarchical lattice.

Let us denote by $U_{d}\subset \C$ the open set of parameters $\lambda$ for which $f_{d,\lambda}$ has an attracting fixed point.
Then
\begin{equation}\label{eq:parametrise U_d}
U_d=\left\{\frac{-\alpha d^d}{(d+\alpha)^{d+1}}\mid |\alpha|<1\right\}.
\end{equation}
Indeed, writing $f=f_{d,\lambda}$, we note that if $x$ is a fixed point of $f$ we have
\[
f^\prime(x)=\frac{-d}{1+x}\frac{\lambda}{(1+x)^d}=\frac{-dx}{1+x}.
\]
Let $\alpha\in \C$.
Then $f^\prime(x)=\alpha$ if and only if $x=\frac{-\alpha}{d+\alpha}$ and consequently,
\[
\lambda=x(1+x)^d=\frac{-\alpha d^d}{(d+\alpha)^{d+1}}.
\]
A fixed point $x = f(x)$ is attracting if and only if $|f^\prime(x)|<1$, which implies the description \eqref{eq:parametrise U_d}. For parameters $\lambda$ in the boundary $\partial U_{\Delta-1}$ the function $f$ has a neutral fixed point, and for a dense set of parameters $\lambda \in \partial U_{\Delta - 1}$ the fixed point is \emph{parabolic}, i.e. the derivative at the fixed point is a root of unity. Classical results from complex dynamical systems allow us to deduce the following regarding the vanishing/non-vanishing of the independence polynomial:
\begin{proposition}\label{prop:regular trees}
Let $\Delta\in \N$ be such that $\Delta\geq 3$. Then
\begin{itemize}
\item[(i)] for all $k\in \N$ and $\lambda\in U_{\Delta-1}$, $Z_{T_{\Delta,k}}(\lambda)\neq 0$;
\item[(ii)] if $\lambda\in \partial U_{\Delta-1}$, then for any open neighborhood $U$ of $\lambda$ there exists $\lambda'\in U$ and $k\in \N$ such that $Z_{T_{\Delta,k}}(\lambda')=0$.
\end{itemize}
\end{proposition}
We note that for $\lambda=\frac{-(\Delta-1)^{\Delta-1}}{\Delta^{\Delta}}$ part (ii) was proved by Shearer \cite{Sh85}; see also \cite{SS05}.
Part (i) follows quickly from elementary results in complex dynamics, but the statements that imply part (ii) are less trivial. The necessary background from the complex dynamical systems, including the proof of Proposition \ref{prop:regular trees} and a counterexample to the multivariate statement of Theorem~\ref{thm:main}, will be discussed in Appendix~\ref{app:complex dynamics}. Note that Proposition~\ref{prop:tight} from the introduction is a special case of Proposition~\ref{prop:regular trees}.

So we can conclude that Sokal's conjecture is already proved for regular trees.
We now move to general (bounded degree) graphs.

\subsection{A recursive procedure for ratios for all graphs}
It will be convenient to have an expression similar to~\eqref{eq:rec tree} for all graphs.
Let $G$ be a graph with fixed vertex $v_0$.
Let $v_1,\ldots,v_d$ be the neighbors of $v_0$ in $G$ (in any order). Set $G_0=G-v_0$ and define for $i=1,\ldots,d$, $G_i:=G_{i-1}-v_{i}$.
Then $G_{d}=G\setminus N[v_0]$.
The following lemma gives recursive relation for the ratios and has been used before over the real numbers in e.g. \cite{LL15}.
\begin{lemma}\label{lem:weitz}
Suppose $Z_{G_i}(\lambda)\neq 0$ for all $i=0,\ldots,d$.
Then
\begin{equation}
R_{G,v_0}=\frac{\lambda_{v_0}}{\prod_{i=1}^d (1+R_{G_{i-1},v_i})}.
\end{equation}
\end{lemma}
\begin{proof}
Let us write
\begin{align*}
\frac{Z_{G-v_0}(\lambda)}{Z_{G\setminus N[v_0]}(\lambda)}&=\frac{Z_{G_0}(\lambda)}{Z_{G_1}(\lambda)}\frac{Z_{G_1}(\lambda)}{Z_{G_2}(\lambda)}\cdots \frac{Z_{G_{d-1}}(\lambda)}{Z_{G_{d}}(\lambda)}=\prod_{i=1}^{d} \frac{Z_{G_{i}}(\lambda)+\lambda_{v_i}Z_{G_{i-1}\setminus N[v_i]}}{Z_{G_i}(\lambda)}
\\
&=\prod_{i=1}^{d}(1+R_{G_{i-1},v_i}),
\end{align*}
where in the second equality we use \eqref{eq:fund recurrence}.
As
\[R_{G,v_0}=\frac{\lambda_{v_0} Z_{G\setminus N[v_0]}(\lambda)}{Z_{G-v_0}(\lambda)}=\frac{\lambda_{v_0}}
{\frac{Z_{G-v_0}(\lambda)}{Z_{G\setminus N[v_0]}(\lambda)}},
\]
the lemma follows.
\end{proof}

As an illustration of Lemma~\ref{lem:weitz} we will now prove a result that shows that $Z_G(\lambda)$ is nonzero as long as the norms and arguments of the $\lambda_v$ are small enough. This result is implied by our main theorem for angles that are much smaller still, but the statement below is not implied by our main theorem, and is another contribution to Sokal's question \cite[Question 2.4]{S1}.
The proof moreover serves as warm up for the proof of our main result.

\begin{theorem}\label{thm:main1}
Let $G=(V,E)$ be any graph of maximum degree at most $\Delta\geq 2$.
Let $\eps>0$ and let $\lambda\in\C^V$ be such that $|\lambda_v|\leq\tan\big(\frac{\pi}{(2+\eps)(\Delta-1)}\big)$, and such that $|\arg(\lambda_v)|<\frac{\eps/2}{2+\eps}\pi$ for all $v\in V$. Then $Z_G(\lambda)\neq 0$.
\end{theorem}

\begin{proof}
\change{Since the independence polynomial is multiplicative over the disjoint union of graphs,} we may assume that $G$ is connected.
Fix a vertex $v_0$ of $G$.
We will show by induction that for each subset $U\subseteq V\setminus\{v_0\}$ we have
\begin{itemize}
\item[(i)]\text{ $Z_{G[U]}(\lambda)\neq 0$,}
\item[(ii)]\text{if $u\in U$ has a neighbor in $V\setminus U$, then $|R_{G[U],u}|<\tan\big(\frac{\pi}{(2+\eps)(\Delta-1)}\big)$,}
\item[(iii)]\text{if $u\in U$ has a neighbor in $V\setminus U$, then $\Re(R_{G[U],u})>0$.}
\end{itemize}
Clearly, if $U=\emptyset$ both (i), (ii) and (iii) are true.
Now suppose $U\subseteq V\setminus \{v_0\}$ and let $H=G[U]$.
Let $u_0\in U$ be such that $u_0$ has a neighbor in $V\setminus U$ ($u_0$ exists as $G$ is connected).
Let $u_1,\ldots,u_d$ be the neighbors of $u_0$ in $H$. Note that $d\leq \Delta-1$.
Define $H_0=H-u_0$ and set for $i=1,\ldots,d$ $H_i=H_{i-1}-u_i$.
Then by induction we know that for $i=0,\ldots,d$, $Z_{H_i}(\lambda)\neq 0$ and for $i\geq 1$, $\Re(R_{H_{i-1},u_i})>0$, implying that $|1+R_{H_{i-1},u_{i}}|\geq 1$.
So by Lemma \ref{lem:weitz} we know that
\[
|R_{H,u_0}|=\frac{|\lambda_{u_0}|}{\prod_{i=1}^d |1+R_{H_{i-1},u_{i}}|}<|\lambda_{u_0}|\leq \tan\bigg(\frac{\pi}{(2+\eps)(\Delta-1)}\bigg),
\]
showing that (ii) holds for $U$.

To see that (iii) holds we look at the angle $\alpha$ that $R_{H,u_0}$ makes with the positive real axis.
It suffices to show that $|\alpha|<\pi/2$.
Since by induction $\Re(R_{H_{i-1},u_i})>0$ and  $|R_{H_{i-1},u_i}|\leq \tan\big(\frac{\pi}{(2+\eps)(\Delta-1)}\big)$, we see that the angle $\alpha_i$ that $1+R_{H_{i-1},u_i}$ makes with the positive real axis satisfies $|\alpha_i|\leq \frac{\pi}{(2+\eps)(\Delta-1)}.$
This implies by Lemma \ref{lem:weitz} that
\[|\alpha|<(\Delta-1)\frac{\pi}{(2+\eps)(\Delta-1)}+\frac{(\eps/2)\pi}{2+\eps}= \pi/2,\]
showing that (iii) holds.

As by (iii), $R_{H,u_0}$ has strictly positive real part and hence does not equal $-1$ we conclude by \eqref{eq:no zero not -1} that $Z_{H}(\lambda)\neq 0$.
So we conclude that (i), (ii) and (iii) hold for all $U\subseteq V\setminus\{v_0\}$.

To conclude the proof, it remains to show that $Z_G(\lambda)\neq 0$.
Let $v_1,\ldots,v_d$ be the neighbors of $v_0$. Let $G_i$, for $i=0,\ldots,d$, be defined as the graphs $H_i$ above.
Then by (i) and (ii) we know that for $i=0,\ldots,d$, $Z_{G_i}(\lambda)\neq 0$ and $\Re(R_{G_{i-1},v_i})>0$ for $i\geq 1$.
So as above we have
%\[
%|R_{G,v_0}|=\frac{|\lambda_{v_0}|}{\prod_{i=1}^d |1+R_{G_{i-1},v_{i}}|}<|\lambda_{u_0}|\leq\tan\bigg( \frac{\pi}{(2+\eps)\Delta-1}\bigg).
%\]
that the angle $\alpha_i$, that $1+R_{G_{i-1},v_{i}}$ makes with the positive real line, satisfies $|\alpha_i|\leq \frac{\pi}{(2+\eps)\Delta-1}$.
So by Lemma \ref{lem:weitz} the absolute value of the argument of $R_{G,v_0}$ is bounded by
\[
\frac{(\eps/2)\pi}{2+\eps}+\Delta\frac{\pi}{(2+\eps)(\Delta-1)}\leq \frac{(2+\eps/2)\pi}{2+\eps}<\pi,
\]
using that $\frac{\Delta}{\Delta-1}\leq 2$.
This implies by \eqref{eq:no zero not -1} that $Z_G(\lambda)\neq0$ and finishes the proof.
\end{proof}

Define for $\lambda\in \C$ and $d\in \N$ the map $F_{d,\lambda}:\widehat \C^{d}\to \widehat \C$ by
\[
(x_1,\ldots,x_{d})\mapsto \frac{\lambda}{\prod_{i=1}^{d}(1+x_i)}.
\]
Given $\epsilon>0$, the proof of Theorem~\ref{thm:main1} consisted mainly of finding a domain $D\subset \widehat \C$ not containing $-1$ such that if $x_1, \ldots ,x_d \in D$, then $F_{d,\lambda}(x_1, \ldots, x_d) \in D$ for all $0 \le d \le \Delta-1$.

To prove Theorem~\ref{thm:main}, we will similarly construct for each $\Delta$ a domain $D$, containing the interval $[0, \lambda_\Delta]$ but not the point $-1$, which is mapped inside itself by $f_{d,\lambda}$ for all $0\leq d\leq \Delta-1$ and all $\lambda$ in a sufficiently small complex neighborhood of the interval $[0,(1-\epsilon)\lambda_\Delta)$. Had these functions $f_{d,\lambda}$ all been strict contractions on the interval $[0, \lambda_\Delta]$, the existence of such a domain $D$ would have been immediate. Unfortunately the functions $f_{d,\lambda}$ are typically not contractions, even for real valued $\lambda$. However, since the positive real line is contained in the basin of an attracting fixed point, it follows from basic theory of complex dynamical systems \cite{M06} that each $f_{d,\lambda}$ is strictly contracting on $[0,\lambda_\Delta)$ with respect to the Poincar\'e metric of the corresponding attracting basin. While these Poincar\'e metrics vary with $\lambda$ and $d$, this observation does give hope for finding coordinates with respect to which all the maps $f_{d,\lambda}$ are contractions.

In the next section we will introduce explicit coordinates with respect to which $f_{\Delta-1,\lambda_\Delta}$ becomes a contraction, and then show that for $d \le \Delta-1$ and $\lambda \in [0, \lambda_\Delta)$ the maps $f_{d, \lambda}$ are all strict contractions with  respect to the same coordinates. We will then utilize these coordinates to give a proof of Theorem~\ref{thm:main} in Section~\ref{sec:proof}.

\section{A change of coordinates}\label{sec:coordinate}

It is our aim in this section to find a coordinate change for each $\Delta\geq 3$ so that the maps $f_{d,\lambda}$ are contractions in these coordinates for any $0\leq d\leq \Delta-1$ and any $0\leq \lambda \leq \lambda_\Delta$.

\subsection{The case $d=\Delta-1$ and $\lambda=\lambda_\Delta$}
We consider the coordinate changes.
$$
z = \varphi_y(x) = \log(1+y\log(1+x)),
$$
with $\change{y>0}$. We note that a similar coordinate change using a double logarithm was used in \cite{LL15}. The best argument for using the specific form above is that it seems to fit our purposes.

Our initial goal is to pick a $y$, depending on $\Delta$ such that the \emph{parabolic} map $f(x):=f_{\Delta-1,\lambda_\Delta}(x)$ becomes a contraction with respect to the new coordinates. Note that we call $f$ parabolic if $\lambda = \lambda_\Delta$.
In this case the fixed point of $f$ is given by
$$
x_\Delta=\frac{1}{\Delta-2} = \frac{1}{d-1},
$$
and has derivative $f^\prime(x_\Delta) = -1$, and is thus parabolic.
In the $z$-coordinates we consider the map
$$
g(z) = g_{\Delta-1,\lambda_\Delta}(z) = \varphi_y \circ f  \circ \varphi_y^{-1}.
$$
\change{Note that the function $\varphi_y:\mathbb R_+ \rightarrow \mathbb R_+$ is bijective, and $\mathbb R_+$ is forward invariant under $f$. It follows that the composition $g$ is well defined on $ \mathbb R_+$.}
\change{We write $z_\Delta:= \varphi_y(x_\Delta)$.}
Then $z_\Delta$ is fixed under $g$, and one immediately obtains $g^\prime(z_\Delta) = -1$. Thus, in order for $|g^\prime| \le 1$ we in particular need that $g^{\prime\prime}(z_\Delta) = 0$.

Let us start by computing $g^\prime$ and $g^{\prime\prime}$.
Writing $x_1 = f(x_0)$ and $z_0 = \varphi_y(x_0)$ we note that
\begin{align}
g^\prime(z_0) = & \varphi_y^\prime(x_1) \cdot f^\prime(x_0) \cdot (\varphi_y^{-1})^\prime(z_0) \nonumber
\\
= & \frac{\varphi_y^\prime(x_1)}{\varphi_y^\prime(x_0)} \cdot f^\prime(x_0)	\nonumber
\\
= & \frac{1+ y\log(1+x_0)}{1+y\log(1+x_1)} \cdot \frac{1+x_0}{1+x_1} \cdot \frac{-d x_1}{1+x_0}\nonumber
\\
= & \frac{1+ y\log(1+x_0)}{1+y\log(1+x_1)} \cdot \frac{-d x_1}{1+x_1}.	\label{eq:g prime}
\end{align}
\change{Now note that
$$
g^{\prime\prime} = \frac{\partial g^\prime}{\partial x_0} \cdot \frac{\partial x_0}{\partial z_0},
$$
and since $ \frac{\partial x_0}{\partial z_0} > 0$, we look for points $z_0$ where $ \frac{\partial g^\prime}{\partial x_0}(z_0) = 0 $.} We obtain
$$
\begin{aligned}
\change {\frac{\partial g^\prime}{\partial x_0}(z_0)}  = & \frac{y/(1+x_0)}{1+y\log(1+x_1)} \cdot \frac{-d x_1}{1+x_1}\\
 + & (1+ y\log(1+x_0)) \cdot \frac{\partial}{\partial x_1} \left( \frac{1}{1+y\log(1+x_1)} \cdot \frac{-d x_1}{1+x_1}\right) \cdot \frac{\partial x_1}{\partial x_0}.
\end{aligned}
$$
\change{By considering $x_1$ as a variable depending on $x_0$, and thus also on $z_0$, the presentation of the calculations here and later in this section becomes significantly more succinct}. Since
\[
\frac{\partial x_1}{\partial x_0}=\frac{-dx_1}{1+x_0},
\]
and since
\begin{equation}\label{eq:partial}
\frac{\partial}{\partial x_1} \left( \frac{1}{1+y\log(1+x_1)} \cdot \frac{-dx_1}{1+x_1}\right) = d \cdot \frac{x_1 y - (1+y\log(1+x_1))}{(1+x_1)^2(1+y\log(1+x_1))^2},
\end{equation}
we obtain
\begin{align}\label{eq:g double prime}
\change {\frac{\partial g^\prime}{\partial x_0}(z_0)}  =& \frac{y/(1+x_0)}{1+y\log(1+x_1)} \cdot \frac{-d x_1}{1+x_1}
\\
&+(1+ y\log(1+x_0))\frac{-d^2x_1}{1+x_0}\cdot \frac{x_1 y - (1+y\log(1+x_1))}{(1+x_1)^2(1+y\log(1+x_1))^2}.	\nonumber
\end{align}

\begin{prop}\label{prop:compute y}
The only value of $y > 0$ for which $g^{\prime\prime}(z_\Delta) = 0$ is given by
$$
y = y_\Delta := \frac{1}{2x_\Delta - \log(1+x_\Delta)}.
$$
\end{prop}
\begin{proof}
Noting that $x_1 = x_0$ and $dx_1/(1+x_1)=1$ when $x_0 = x_\Delta$, we obtain
$$
\frac{\partial g^\prime}{\partial x_0}(z_\Delta) = d \cdot \frac{1+y\log(1+x_\Delta) - 2x_\Delta y}{(1+y\log(1+x_\Delta))(1+x_\Delta)^2}.
$$
Thus $g^{\prime\prime}(z_\Delta) = 0$ if and only if
$$
y=y_\Delta := \frac{1}{2x_\Delta - \log(1+x_\Delta)}.
$$
\end{proof}

From now \change{on} we assume that $y = y_\Delta$.

\begin{corollary}\label{cor:g prime is bounded}
We have that $|g^\prime(z)| \le 1$ for all $z \ge 0$.
\end{corollary}
\begin{proof}
Since
$$
\lim_{z \rightarrow +\infty} |g^\prime(z)| = 0,
$$
it suffices to show that $|g^\prime(0)| < 1$, which follows if we show that $g^{\prime\prime}(0)<0$\change{, for which it is sufficient to show that $\frac{\partial g^\prime}{\partial x_0}(0)<0$}.

Plugging in $x_0=0$ in \eqref{eq:g double prime} we get
$$
\change{\frac{\partial g^\prime}{\partial x_0}(0)} = \frac{dx_1\left(d - y(1+x_1)\right)(1+y\log(1+x_1) - dx_1 y)}{(1+y\log(1+x_1))^2(1+x_1)^2},
$$
with $x_1 = f(0) = \lambda$. Hence we can complete the proof by showing that
\begin{equation}
d - y(1+\lambda) = d - y - y\lambda < 0.\label{eq:g'(0)<1}
\end{equation}
Using that $1/y=2/(d-1)-\log(d/(d-1))$ we observe that
\[
1/y<\frac{1}{d-1}+\frac{1}{2(d-1)^2}
\]
and hence $y > \frac{(d-1)^2}{d-1/2}$ .
From this we obtain
\begin{align*}
d - y - y\lambda &<\frac{d(d-1/2)-(d-1)^2-(d-1)(d/(d-1))^d)}{d-1/2}
\\
&<\frac{d(d-1/2)-(d-1)^2-(d-1)(1+d/(d-1))}{d-1/2}=\frac{-d/2}{d-1/2}<0,
\end{align*}
which completes the proof.
\end{proof}

In particular it follows that for all $x \ge 0$ we have that $f^{\circ n}(x) \rightarrow x_\Delta$.

\subsection{Smaller values of $\lambda$ and $d$}
We now consider the case where $\lambda < \lambda_\Delta$, and the map $f$ has degree $d\leq \Delta-1$.
We again consider the map
\[
g_{d,\lambda}(z)=\varphi_{y}\circ f_{d,\lambda}\circ \varphi^{-1}_y.
\]
Again we will often just write $g$ instead of $g_{d,\lambda}$.
Our goal is to show that \change{$|g^\prime(z_0)| < 1$ for all $z_0 \geq 0$}.

To do so we will consider $g^\prime$ as a function of $\lambda,d$ and $z_0$.
We first look at the case where $\lambda$ is fixed and $d$ is varying.

\begin{lemma}\label{lem:degree}
Let $\Delta\in \N$ with $\Delta\geq 3$.
Let $0\leq \lambda\leq \lambda_\Delta$ and let $d\in\{0,1,\ldots,\Delta-1\}$.
Let $z_0\geq 0$ be such that $g^{\prime\prime}_{d,\lambda}(z_0)=0$. Then we have $0\geq g_{d,\lambda}^\prime(z_0)\geq g_{\Delta-1,\lambda}^\prime(z_0)$.
\end{lemma}
\begin{proof}
We will consider the derivative of $g^\prime$ with respect to $d$ in the points $z_0$ where $g^{\prime\prime}(z_0)=0$.
By \eqref{eq:g double prime}, $g^{\prime\prime}(z_0)$ is a multiple of
$$
\frac{y}{1+y\log(1+x_1)} \cdot \frac{1}{1+x_1} + -d (1+y\log(1+x_0)) \cdot \frac{1+y\log(1+x_1)-x_1y}{(1+x_1)^2(1+y\log(1+x_1))^2}.
$$
As $g^{\prime\prime}(z_0) =0$, we obtain
\[
y(1+x_1)(1+y\log(1+x_1)) = d (1+y\log(1+x_0)) \cdot (1+y\log(1+x_1) - x_1 y).
\]
In particular we get that
\begin{equation}\label{eq:positive}
1 + y \log(1+x_1) - x_1 y > 0,
\end{equation}
and
\begin{equation}\label{eq:elimination}
d\log(1+x_0) = \frac{(1+x_1)(1+y\log(1+x_1))}{1+y\log(1+x_1)-x_1y} - \frac{d}{y}.
\end{equation}
Now notice that by \eqref{eq:g prime} we have that $\frac{\partial}{\partial d} g^\prime$ is a positive multiple of
\[
\frac{-x_1}{(1+x_1)(1+y\log(1+x_1)}+\frac{\partial x_1}{\partial d}\cdot \frac{\partial}{\partial x_1} \left( \frac{1}{1+y\log(1+x_1)} \cdot \frac{-dx_1}{1+x_1}\right),
\]
which by \eqref{eq:partial} is a positive multiple of
$$
- (1+x_1) (1 + y \log(1+x_1)) + d\log(1+x_0) \cdot (1 + y \log(1+x_1) - x_1 y).
$$
When we plug in equation \eqref{eq:elimination} to eliminate $x_0$ from this expression, we note that the term $(1+x_1)(1+y\log(1+x_1))$ cancels and we obtain that $\frac{\partial}{\partial d}g^\prime$ is a positive multiple of
$$
- \frac{d}{y} \left(1 + y\log(1+x_1) - x_1 y\right),
$$
which is negative as observed in \eqref{eq:positive}.

So, we see that as we decrease $d$ the value of $g^\prime(z_0)$ increases and hence it follows that $0\geq g_{d,\lambda}^\prime(z_0) \geq  g_{\Delta-1,\lambda}^\prime(z_0)$, as desired.
\end{proof}

We next compute the derivative of $g^\prime$ with respect to $\lambda$.
Note that $x_1$ depends  on $\lambda$, but $x_0$ does not, hence
$$
\begin{aligned}
\frac{\partial g^\prime}{\partial \lambda} (z_0) = (1+ y \log(1+x_0)) \cdot \frac{\partial}{\partial \lambda} \left(\frac{-dx_1}{(1+x_1)(1+y\log(1+x_1))}\right)\\
= (1+ y \log(1+x_0)) \cdot \frac{\partial x_1}{\partial \lambda} \cdot \frac{\partial}{\partial x_1}\left(\frac{-dx_1}{(1+x_1)(1+y\log(1+x_1))}\right) .
\end{aligned}
$$
Thus $\frac{\partial g^\prime}{\partial \lambda} (z_0) = 0$ if and only if
$$
\frac{\partial}{\partial x_1}\left(\frac{-dx_1}{(1+x_1)(1+y\log(1+x_1))}\right) = 0,
$$
which, by \eqref{eq:partial} is the case if and only if
\begin{equation}\label{eq:second}
x_1 y - (1+y\log(1+x_1)) = 0.
\end{equation}

\begin{lemma}\label{lem:decreasing}
Let $\Delta\geq 5$. For any $\lambda\leq \lambda_\Delta$ and $0\leq d\leq \Delta-1$, we have
$$
x_1 y - (1+y\log(1+x_1))<0.
$$
In particular $g^\prime(z_0)$ is decreasing in $\lambda$ for any $z_0\ge 0$.
\end{lemma}
\begin{proof}
We note that $x_1 y - (1+y\log(1+x_1))$ is increasing in $x_1$ for $x_1>0$.
So it suffices to plug in $\lambda=\lambda_\Delta$ and $x_0=0$, that is, plug in $x_1=\lambda_\Delta$.
Note that this makes it independent of $d$.

Plugging in $x_1=\lambda_\Delta$ we get
\[
\lambda_\Delta y-(1+y\log(1+\lambda_\Delta)=y(\lambda_\Delta-(1/y+\log(1+\lambda_\Delta)),
\]
So as $y>0$ is suffices to show
\begin{equation}\label{eq:no solution}
c(\Delta):=\lambda_\Delta-(1/y+\log(1+\lambda_\Delta)<0.
\end{equation}
By a direct computer calculation, we obtain the following approximate values for $c(\Delta)$ for $\Delta\in\{5,6,7\}$:
\[
\begin{array}{|l|c|c|c|c|}\hline \Delta &5&6&7
\\
\hline
c(\Delta)&-0.0450&-0.0809&-0.0887
\\
\hline
\end{array}
\]
and we conclude that \eqref{eq:no solution} holds for $\Delta\in\{5,6,7\}$.

Using that $x-x^2/2\leq \log(1+x)\leq x$ for all $x\geq 0$, we obtain
\begin{align*}
\lambda_\Delta-(1/y+\log(1+\lambda_\Delta)\leq& \lambda_\Delta-(x_\Delta +\lambda_\Delta-\lambda_\Delta^2/2)
\\
=&\lambda_\Delta^2/2-\frac{1}{\Delta-2}.
\end{align*}
Using that
\[\lambda_\Delta=\frac{\Delta-1}{(\Delta-2)^2}\bigg(\frac{\Delta-1}{\Delta-2}\bigg)^{\Delta-2}\leq \frac{e(\Delta-1)}{(\Delta-2)^2},\]
we obtain that
\begin{equation}\label{eq:bound for dg4}
\lambda_\Delta-(1/y+\log(1+\lambda_\Delta))\leq \frac{e^2(1+\frac{1}{\Delta-2})^2-2(\Delta-2)}{2(\Delta-2)^2}.
\end{equation}
Since the right-hand side of \eqref{eq:bound for dg4} is negative for $\Delta=8$ and since the numerator is clearly decreasing in $\Delta$, we conclude that \eqref{eq:no solution} is true for all $\Delta\geq 8$. This concludes the proof.
\end{proof}

\begin{lemma} \label{lem:case 3,4}
Let $\Delta\in\{3,4\}$.
Let $z_0>0$ and $\lambda_0>0$ be such that
$$
\frac{\partial}{\partial \lambda}g^\prime_{\Delta-1,\lambda}(z_0) = 0
$$
for $\lambda = \lambda_0$. Then $g^\prime_{\Delta-1,\lambda_0}(z_0)\geq -0.92$.
\end{lemma}
\begin{proof}
By assumption we have $\frac{\partial g^\prime(z_0)}{\partial \lambda} = 0$. Thus \eqref{eq:second} implies that
\begin{equation}\label{eq:solution}
x_1 y = 1 + y \log(1 + x_1),
\end{equation}
This implies that for $x_1$ to be a solution to \eqref{eq:solution}, we need that $x_1\geq x_\Delta$.
Indeed suppose that $x_1<x_\Delta$. Then we have from \eqref{eq:solution} that
\[
x_1y=1+y\log(1+x_1)>1+yx_1-yx_1^2/2,
\]
from which we obtain $yx_1^2>2$.
However, as $y<1/x_\Delta$ we have $yx_1^2<yx_\Delta^2<x_\Delta<2$, a contradiction.

Now \eqref{eq:solution} combined with \eqref{eq:g prime} gives
\begin{equation}\label{eq:cases g prime}
g^\prime(z_0)=\frac{1+y\log(1+x_0)}{yx_1}\cdot \frac{-(\Delta-1)x_1}{1+x_1} = \frac{-(\Delta-1)(1+y\log(1+x_0))}{y(1+x_1)}.
\end{equation}
%Observe that $|g_{d,\lambda^\star}^\prime(z_0)|$ is increasing in $x_0$, and decreasing in $x_1$.
Now recall that $y=y_\Delta$ satisfies
\begin{equation}
2x_\Delta y = 1+y\log(1+x_\Delta)  .\label{eq:equation xd}
\end{equation}
Now using that $x_1\geq x_\Delta$ and by combining \eqref{eq:solution} and \eqref{eq:equation xd} we obtain
\[
x_1=\frac{1+y\log(1+x_1)}{y}=2x_\Delta\left(\frac{1+y\log(1+x_1)}{1+y\log(1+x_\Delta)}\right)\geq 2x_\Delta.
\]
Using this we obtain
\begin{align*}
x_1=&2x_\Delta\left(\frac{1+y\log(1+x_1)}{1+y\log(1+x_\Delta)}\right)\geq 2x_\Delta\left(\frac{1+y\log(1+2x_\Delta)}{1+y\log(1+x_\Delta)}\right)
\\
=&2x_\Delta\left( \frac{1+\frac{\log(1+2x_\Delta)}{2x_\Delta-\log(1+x_\Delta)}}{1+\frac{\log(1+x_\Delta)}{2x_\Delta-\log(1+x_\Delta)}}\right)=2x_\Delta+\log\left(\frac{1+2x_\Delta}{1+x_\Delta}\right)=\alpha_\Delta x_\Delta,
\end{align*}
where $\alpha_3=2+\log(3/2)\approx 2.405$, and where $\alpha_4=2+2\log(4/3))\approx 2.575$.
This then implies that
\[
1+x_0\leq (\lambda_\Delta/x_1)^{1/(\Delta-1)}\leq \alpha_\Delta^{-1/(\Delta-1)}(1+x_\Delta).
\]
Since \eqref{eq:cases g prime} is decreasing in $x_0$ and increasing in $x_1$, we can plug in $x_0=\alpha_\Delta^{-1/(\Delta-1)}(1+x_\Delta)$ and $x_1=\alpha_\Delta x_\Delta$ to obtain
$$
\begin{aligned}
g_{\Delta-1,\lambda_0}^\prime(z_0) > & \frac{-(\Delta-1)(1+y\log(\alpha_\Delta^{-1/(\Delta-1)}(1+x_\Delta))}{y(1+\alpha_\Delta x_\Delta)}
=
\frac{-2(\Delta-1)x_\Delta y + y\log (\alpha_\Delta)}{y(1+\alpha_\Delta x_\Delta)}
\\
 = &\frac{-2(\Delta-1)/(\Delta-2) + \log(\alpha_\Delta)}{(\Delta-2+\alpha_\Delta)/(\Delta-2)}
\approx\left\{ \begin{array}{ll} -0.9168\text{ if } \Delta=3,
																														 \\ -0.8979\text{ if } \Delta=4.
																														 \end{array}\right.
\end{aligned}
$$
This finishes the proof.
\end{proof}

We can now finally show that the coordinate changes works for all values of the parameters we are interested in.

\begin{prop} \label{prop:explicit bound dg}
Let $\Delta\geq 3$ and let $\eps>0$. Then there exists $\delta>0$ such that if $0\leq \lambda<(1-\eps)\lambda_\Delta$, then $|g^\prime_{d,\lambda}(z_0)|<1-\delta$ for all $z_0\geq 0$ and $d\in \{0,1,\ldots,\Delta-1\}$.
\end{prop}

\begin{proof}
%First of all, we may assume that $d\geq 1$, for else $g^\prime(z_0)=0$ for any $z_0$.
Let $J=[0,(1-\eps)\lambda_\Delta]$ and let
\[
M:=\min_{z_0\geq 0,\lambda\in J,d=0,\ldots, \Delta-1} g^{\prime}_{d,\lambda}(z_0).
\]
As for any $\lambda\in J$ we have that $\lim_{z_0\to \infty}g^\prime_{d,\lambda}(z_0)=0$ and as $g^{\prime\prime}(0)<0$ by the proof of Corollary~\ref{cor:g prime is bounded} (which remains valid as \eqref{eq:g'(0)<1} is decreasing in $d$) it follows that we may assume that $M$ is attained at some triple $(z_0,\lambda_0,d)$ with $z_0>0$, $\lambda_0\in J$ and $d\in \{0,\ldots,\Delta-1\}$.
This then implies that $g_{d,\lambda_0}^{\prime\prime}(z_0)=0$ and hence by Lemma~\ref{lem:degree} we know that $g^{\prime}_{d,\lambda_0}(z_0)\geq g^{\prime}_{\Delta-1,\lambda_0}(z_0)$,
that is, we have that $d=\Delta-1$.

If $g^\prime_{d,\lambda_0}(z_0)$ attains its minimum (as a function of $\lambda$) at some $\lambda<\lambda_\Delta$, then $\frac{\partial}{\partial \lambda} g^{\prime}(z_0)=0$. So by Lemma~\ref{lem:decreasing} we know that $\Delta\in \{3,4\}$.
Then Lemma~\ref{lem:case 3,4} implies that  $M\geq -0.92$.
So we may assume that $g^{\prime}$ is strictly decreasing as a function of $\lambda$ on $[0,\lambda_\Delta]$.
This then implies that $\lambda_0=(1-\eps)\lambda_\Delta$ and so there exists $\delta>0$ (and we may assume $\delta<0.08$)
such that
\[
M=g_{d,\lambda_0}^\prime(z_0)> (1-\delta) g_{d,\lambda_\Delta}^\prime(z_0)\geq -1+\delta,
\]
where the last inequality is by Corollary~\ref{cor:g prime is bounded}.
This finishes the proof.
\end{proof}

\section{Proof of Theorem~\ref{thm:main}}\label{sec:proof}

Our proof will essentially follow the same pattern as the proof of Theorem~\ref{thm:main1}, but instead of working with the function $F_{d,\lambda}$ we now need to work with a conjugation of $F_{d,\lambda}$.
Let $\Delta\geq 3$.
Recall from the previous section the function $\varphi:\R_{+}\to \R_{+}$ defined by $z = \varphi(x) = \log(1+y\log(1+x))$, with $y=y_\Delta$.
We now extend the function $\varphi$ to a neighborhood $V \subset \C$ of $\R_+$ by taking the branch for both logarithms that is real for $x >0$. By making $V$ sufficiently small we can guarantee that $\varphi$ is invertible. Now define for $d=0,\ldots, \Delta-1$, the map $G_{d,\lambda}: \varphi(V)^d \to \widehat \C$ by
\[
(z_1,\ldots,z_d)\mapsto \varphi\left(\frac{\lambda}{\prod_{i=1}^d 1+\varphi^{-1}(z_i)}\right).
\]
For a set $A\subset \C$ and $\eps>0$ we write $\mathcal{N}(A,\eps):=\{z\in \C\mid |z-a|<\eps \text{ for some }a\in A\}$.
Now define for $\eps>0$ the set $D(\eps)\subset \C$ by
\[
D(\eps):=\mathcal{N}([0,\varphi(\lambda_\Delta)],\eps).
\]
We collect a very useful property:
\begin{lemma}\label{lem:invariant domain}
Let $\Delta\geq 3$ and let $\eps>0$.
Then there exist $\eps_1,\eps_2>0$ such that for any $\lambda\in \Lambda(\eps_2):=
%\{\lambda\in \mathcal{N}([0,(1-\eps)\lambda_\Delta ],\eps_2)\mid \Re(\lambda)\geq 0\}$
\change{\mathcal{N}([0,(1-\eps)\lambda_\Delta ],\eps_2)}$, any $d=0,\ldots,\Delta-1$ and $z_1,\ldots,z_d\in D(\eps_1)$ we have $G_{d,\lambda}(z_1,\ldots,z_d)\in D(\eps_1)$.
\end{lemma}
\begin{proof}
We first prove this for the special case that $z_1=z_2=\ldots=z_d=z$.
In this case we have $G_{d,\lambda}(z_1,\ldots,z_d)=g_{d,\lambda}(z)$.
By Proposition~\ref{prop:explicit bound dg} we know that there exists $\delta>0$ such that for any $d=0,\ldots,\Delta-1$ we have
\[
|g^\prime_{d,\lambda}(z)|<1-\delta \text{ for all } \lambda\in [0,(1-\eps)\lambda_{\Delta}] \text{ and } z\in [0,\varphi(\lambda_\Delta)].\]
By continuity of $g^\prime$ as a function of $z$ and $\lambda$ there exists $\eps_1,\eps_2>0$ such that for all $d=0,\ldots,\Delta-1$ and each $(z,\lambda)\in D(\eps_1)\times \change{\Lambda(\eps_2)}$
%\mathcal{N}([0,(1-\eps)\lambda_{\Delta}],\eps_2)$
we have
\[
|g_{d,\lambda}^\prime(z)|\leq 1-\delta/2.
\]
We may assume that $\eps_2$ is small enough so that for any $d$,
\[
\sup_{\lambda\in \Lambda\change{(\eps_2)},z\in [-\eps_1,\varphi (\lambda_\Delta)]}\left|\frac{\partial}{\partial \lambda}g_{d,\lambda}(z)\right|\leq \frac{\delta \eps_1}{2\eps_2},
\]
%We may additionally assume that $\eps_2$ is small enough so that $\eps_2g_{d,\lambda_\Delta}(-\eps_1)<\delta\eps_1/2$ for all $d$.
Fix now $\lambda\in \Lambda\change{(\eps_2)}$ and $d$ and let $z\in D(\eps_1)$.
Let $z'\in [0,\varphi(\lambda_\Delta)]$ be such that $|z-z'|< \eps_1$ \change{and let $\lambda'\in [0,(1-\eps)\lambda_\Delta))$ be such that $|\lambda-\lambda'|<\eps_2$}
Then
\begin{align*}
|g_{d,\lambda}(z)-g_{d,\lambda'}(z')|\leq &|g_{d,\lambda}(z)-g_{d,\lambda}(z')|+|g_{d,\lambda}(z')-g_{d,\lambda'}(z')|
\\
< & (1-\delta/2)\eps_1+\eps_1\delta/2<\eps_1,
\end{align*}
implying that the distance of $g_{d,\lambda}(z)$ to $[0,\varphi(\lambda_\Delta)]$ is at most $\eps_1$, as $g_{d,\lambda'}(z')\in [0,\varphi(\lambda_\Delta)]$. Hence $g_{d,\lambda}(z)\in D(\eps_1)$, which proves the lemma for $z_1=z_2=\ldots=z_d=z$.

For the general case fix $d$, let $\lambda\in \Lambda(\change{\eps_2})$ and consider $x=\prod_{i=1}^d (1+\varphi^{-1}(z_i))$ for certain $z_i\in D=D(\eps_1)$.
We want to show that $x=\prod_{i=1}^d (1+\varphi^{-1}(z)) = (1+\varphi^{-1}(z))^d$ for some $z\in D$.
First of all note that
\[1+\varphi^{-1}(z_i)= \exp\left(\frac{\exp(z_i)-1}{y}\right).\]
Then
\[x=\exp\left( \sum_{i=1}^d \left(\frac{\exp(z_i)-1}{y}\right)\right),\]
which is equal to $(1+\varphi^{-1}(z))^d$ for some $z\in D$ provided
\begin{equation}
\frac{1}{d}\sum_{i=1}^d \exp(z_i)=\exp(z),\label{eq:condition}
\end{equation}
for some $z\in D$. \change{Consider the image of $D$ under the exponential map. $D$ is a smoothly bounded domain whose boundary consist of two arbitrarily small half-circles and two parallel horizontal intervals. Recall that the exponential imagine of a disk of radius less than $1$ is strictly convex, a fact that can easily be checked by computing that the curvature of its boundary has constant sign. Therefore $\exp(D)$ is a smoothly bounded domain whose boundary consists of two radial intervals and two strictly convex curves, hence $\exp(D)$ must also be convex. See Figure \ref{fig:convexity} for a sketch of the domain $D$ and its image under the exponential map.} %. and consisting of two arbitrarily small disks, and }Since the exponential function changes curvature by at most a bounded additive amount, it follows that the exponential image of any sufficiently small disk is strictly convex. Therefore the image of the set $D$ under the exponential map is convex.
It follows that the convex combination $\frac{1}{d}\sum_{i=1}^d \exp(z_i)$ is contained in the image of $D$.
In other words, there exists $z\in D$ such that \eqref{eq:condition} is satisfied.
This now implies that $G_{d,\lambda}(z_1,\ldots,z_d)=g_{d,\lambda}(z)\in D$, as desired.
\end{proof}

\begin{figure}
\label{fig:convexity}
\hspace{.4in} \includegraphics[width=5in]{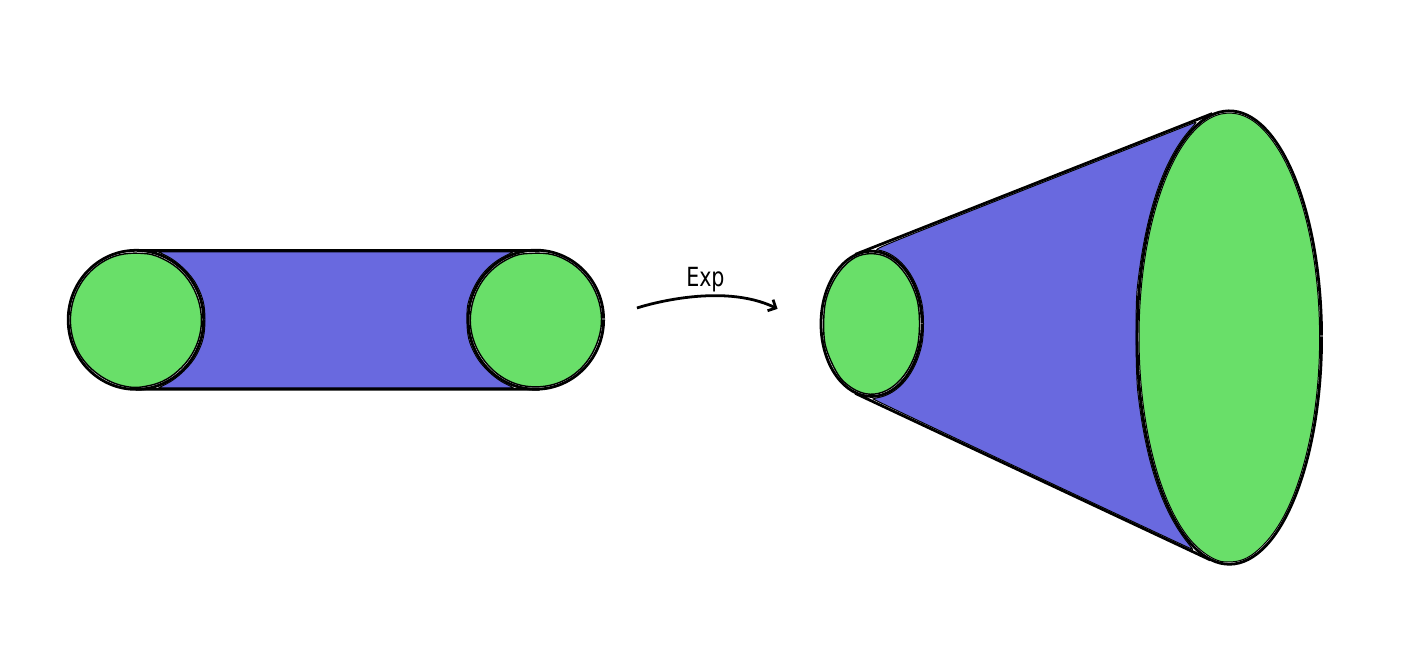}
\caption{The set $D$ and its image under the exponential map.}
\end{figure}

\subsection{Proof of Theorem~\ref{thm:main}}
We first state and prove a more precise version of Theorem~\ref{thm:main} for the multivariate independence polynomial:
\begin{theorem}\label{thm:precise}
Let $\Delta \in \mathbb N $ with $\Delta \geq 3$. Then for any $\eps>0$ there exists $\delta>0$ such \change{that} for any graph $G=(V,E)$ of maximum degree at most $\Delta$ and any $\lambda\in \C^V$ satisfying $\lambda_v \in \mathcal{N}([0,(1-\eps)\lambda_\Delta),\delta)$ for each $v\in V$, we have that $Z_G( \lambda)\neq 0$.
\end{theorem}
\begin{proof}
Let $\eps_1$ and $\eps_2$ be the two constants from Lemma~\ref{lem:invariant domain}, \change{where $\eps_1$ is chosen sufficiently small.}
Let $D=D(\eps_1)$ and let $\delta=\eps_2$.
Let $G$ be a graph of maximum degree at most $\Delta$.
\change{Since the independence polynomial is multiplicative over the disjoint union of graphs,} we may assume that $G$ is connected.
Fix a vertex $v_0$ of $G$.
We will show by induction that for each subset $U\subseteq V\setminus\{v_0\}$ we have
\begin{itemize}
\item[(i)]\text{ $Z_{G[U]}(\lambda)\neq 0$,}
\item[(ii)]\text{if $u\in U$ has a neighbor in $V\setminus U$, then $\varphi(R_{G[U],u})\in D$},
%\item[(iii)]\text{if $u\in U$ has a neighbor in $V\setminus U$, then $\Re(R_{G[U],u})>0$.}
\end{itemize}
Clearly, if $U=\emptyset$, then both (i) and (ii) are true.

Now suppose $U\subseteq V\setminus \{v_0\}$ is nonempty and let $H=G[U]$.
Let $u_0\in U$ be such that $u_0$ has a neighbor in $V\setminus U$ ($u_0$ exists as $G$ is connected).
Let $u_1,\ldots,u_d$ be the neighbors of $u_0$ in $H$. Note that $d\leq \Delta-1$.
Define $H_0=H-u_0$ and set for $i=1,\ldots,d$ $H_i=H_{i-1}-u_i$.
Then by induction we know that for $i=0,\ldots,d$, $Z_{H_i}(\lambda)\neq 0$ and so the ratios $R_{H_{i-1},u_{i}}$ are well defined for $i\geq 1$ and by induction they satisfy $\varphi(R_{H_{i-1},u_{i}})\in D$.
By Lemma~\ref{lem:weitz}
\[
R_{H,u_{0}}=\frac{\lambda_{u_0}}{\prod_{i=1}^d(1+R_{H_{i-1},u_{i}})}.
\]
Since $\varphi(R_{H_{i-1},u_{i}})\in D$ for $i=1,\ldots,d$, we have by Lemma~\ref{lem:invariant domain} that $\varphi(R_{H,u_{0}})\in D$.
From this we conclude that $R_{H,u_{0}}\neq -1$, as $-1\notin \varphi^{-1}(D)$.
So by \eqref{eq:no zero not -1} $Z_H(\lambda)\neq 0$.
This shows that (i) and (ii) hold for all subsets $U\subseteq V\setminus\{v_0\}$.

To conclude the proof we need to \change{show that $Z_G(\lambda)\neq 0$}.
%look at the vertex $v_0$.
Let $v_1,\ldots,v_d$ be the neighbors of $v_0$ (in any order). Define $G_0=G-v_0$ and set for $i=1,\ldots,d$ $G_i=G_{i-1}-v_i$.
Then by \change{(i)}
%induction
we know that for $i=0,\ldots,d$, $Z_{G_i}(\lambda)\neq 0$ and so the ratios $R_{G_{i-1},v_{i}}$ are well defined for $i\geq 1$ and by %induction
\change{(ii)}
they satisfy $\varphi(R_{G_{i-1},v_{i}})\in D$. Write for convenience $z_i=R_{G_{i-1},v_{i}}$ for $i=1,\ldots,d$.
Then, by the same reasoning as above, we have
\[R_{G,v_{0}}(1+z_d)=\frac{\lambda_{v_0}}{\prod_{i=1}^{d-1}(1+z_i)}\in \varphi^{-1}(D).\]
This implies that $R_{G,v_{0}}$ is not equal to $-1$, for if this were the case, we would have
$\change{-1\in z_d+\varphi^{-1}(D)}$. However, $\change{z_d\in \varphi^{-1}(D)}$ \change{and for $\eps_1$ small enough, $\varphi^{-1}(D)$ will have real part bounded away from $-1/2$,  a contradiction.}
We conclude that $Z_G(\lambda)\neq 0$.
\end{proof}
Theorem~\ref{thm:main} is now an easy consequence.
\begin{proof}[Proof of Theorem~\ref{thm:main}]
Let for $\eps>0$, $\delta(\eps)$ be the associated $\delta>0$ from Theorem~\ref{thm:precise}.
Consider a sequence $\epsilon_i \rightarrow 0$ and define
\[
D_\Delta:=\bigcup \mathcal{N}([0,(1-\eps_i)\lambda_\Delta),\delta(\eps_i)).
\]
The set $D_\Delta$ is clearly open and contains $[0,\lambda_\Delta)$.
Moreover, for any graph $G$ of maximum degree at most $\Delta$ and $\lambda\in D_\Delta$ we have $Z_G(\lambda)\neq 0$, as $\lambda\in \mathcal{N}([0,(1-\eps)\lambda_\Delta),\delta(\eps))$ for some $\eps>0$.
\end{proof}
Let us recall that the literal statement of Theorem~\ref{thm:main} is false in the multivariate setting as we will prove in the appendix.
%More precisely, we show in Theorem~\ref{thm:} below that for each neighborhood $D$ of $[0,\lambda_\Delta)$ there exists a graph $G=(V,E)$ of maximum degree at most $\Delta$ and $\lambda\in D^V$ such that $Z_G(\lambda)=0$.
However, by the same reasoning as above we do immediately obtain the following.
\begin{corollary}\label{cor:mult}
Let $\Delta \in \mathbb N $ with $\Delta \geq 3$, and let $n \in \mathbb N$. Then there exists a complex domain $D_\Delta$ containing $[0,\lambda_\Delta)^\change{n}$ such that for any graph $G=(V,E)$ \change{with $V = \{1,\ldots , n\}$} of maximum degree at most $\Delta$ and any $\lambda \in D_\Delta$, we have that $Z_G( \lambda)\neq 0$.
\end{corollary}
\change{We remark that the difference between Corollary~\ref{cor:mult} and Theorem~\ref{thm:main} is subtle. The set $D_\Delta$ is chosen of the form
$$
D_\Delta:=\bigcup \mathcal{N}([0,(1-\eps_i)\lambda_\Delta),\delta(\eps_i))^n,
$$
as above. In particular the set $D_\Delta$ is not of the form $D^n$ for some open set $D$ containing $[0,\lambda_\Delta)$, hence in this sense it is not a literal generalization of Theorem~\ref{thm:main}.} %in contrast with the univariate case.
%In Corollary~\ref{cor:mult} the domain $D\subset \C^V$ may depend on the graph $G$ while fixing $D'\subset \C$ containing $[0,\lambda_\Delta)$ and setting $D=D'^V$ this is not the case.}

\section{Concluding remarks and questions}\label{sec:conclusion}
In this paper we have shown that Sokal's conjecture is true.
By results from \cite{PR16} this gives as a direct application the existence of an efficient algorithm (different than Weitz's algorithm \cite{W6}) for approximating the independence polynomial at any fixed $0<\lambda<\lambda_\Delta$.
By a result of Sly and Sun \cite{SS14} it is known that unless NP=RP there does not exist an efficient approximation algorithm for computing the independence polynomial at $\lambda>\lambda_\Delta$ for graphs of maximum degree at most $\Delta$.
Very recently it was shown by Galanis, Goldberg and  \v{S}tefankovi\v{c} \cite{GGS16}, building on locations of zeros of the independence polynomial for certain trees, that it is NP-hard to approximate the independence polynomial at $\lambda<\frac{-(\Delta-1)^{\Delta-1}}{\Delta^\Delta}$ for graphs of maximum degree at most $\Delta$.
Recall from Proposition~\ref{prop:regular trees} that at any $\lambda$ contained in
\[
U_{\Delta-1}=\left\{\lambda_\Delta(\alpha)=\frac{-\alpha (\Delta-1)^{\Delta-1}}{(\Delta-1+\alpha)^\Delta}\mid |\alpha|< 1\right\},
\]
the independence polynomial for regular trees does not vanish and that for any $\lambda\in \partial(U_{\Delta-1})$ there exists $\lambda'$ arbitrarily close to $\lambda$ for which there exists a regular tree $T$ such that $Z_T(\lambda')=0$.
This naturally leads two the following two questions.
\begin{question}
Let $\alpha\in \C$ be such that $|\alpha|>1$. Let $\eps>0$ and let $\Delta\in \N$. Is it true that it is NP-hard to compute an $\eps$-approximation\footnote{By an $\eps$-approximation of $Z_G(\lambda)$ we mean a nonzero number $\zeta\in \C$ such that $e^{-\eps}\leq |Z_G(\lambda)/\zeta| \leq e^{\eps}$ \change{and such that the angle between $Z_G(\lambda)$ and $\zeta$ is at most $\eps$}.} of the independence polynomial at $\lambda_\Delta(\alpha)$ for graphs $G$ of maximum degree at most $\Delta$?
\end{question}
\change{This question has recently been answered positively, in a strong form, by Bez\'akov\'a, Galanis, Goldberg, and \v{S}tefankovi\v{c}~\cite{BGGS17}. They in fact showed that it is even \#P hard to approximate the independence polynomial at non-positive $\lambda$ contained in the complement of the closure of $U_{\Delta-1}$.}
\begin{question}
Is it true that for any graph $G$ of maximum degree at most $\Delta\geq 3$ and any $\alpha \in \C$ with $|\alpha|<1$ one has $Z_G(\lambda_\Delta(\alpha))\neq 0$?
The same question is also interesting for the multivariate independence polynomial.
\end{question}
We note that if \change{this question too has a positive answer,} it would lead to a complete understanding of the complexity of approximating the independence polynomial of graphs at any complex number $\lambda$ in terms of the maximum degree.

%\newpage
\quad \\
\noindent {\LARGE \bf Appendix}
\begin{appendix}
%\appendix

\section{Parabolic bifurcations in complex dynamical systems, and Proposition~\ref{prop:regular trees}}\label{app:complex dynamics}

The proof of Proposition~\ref{prop:regular trees} follows from results well known to the complex dynamical systems community, but not easily found in textbooks. In this appendix we give a short overview of the results needed, and outline how Proposition~\ref{prop:regular trees} can be deduced from these results. The presentation is aimed at researchers who are not experts on parabolic bifurcations. Details of proofs will be given only in the simplest setting. Readers interested in working out the general setting are encouraged to look at the provided references.

We consider iteration of the rational function
$$
f_\lambda(z) = \frac{\lambda}{(1+z)^d},
$$
where $\lambda \in \mathbb C$ and $d \ge 2$. We note that $f_\lambda$ has two critical %values
\change{points}, $-1$ and $\infty$, and that $f_\lambda(-1) = \infty$.

\begin{lemma}\label{lemma:critical}
If $f_\lambda$ has an attracting or parabolic periodic orbit $\{x_1, \ldots , x_k\}$, then the orbits of $-1$ and $\infty$ both converge to this orbit.
\end{lemma}

This statement is the immediate consequence of the following classical result, which can for example be found in \cite{M06}.

\begin{theorem}
Let $f$ be a rational function of degree $d \ge 2$ with an attracting or parabolic cycle. Then the corresponding immediate basin must contain at least one critical point.
\end{theorem}

Let us say a few words about how \change{ to prove} this result in the parabolic case. Recall that a period orbit is called \emph{parabolic} if its multiplier, the derivative in case of a fixed point, is a root of unity. We consider the model case, where $0$ is a parabolic fixed point with derivative $1$, and $f$ has the form
$$
z_1 = z_0 - z_0^2 + h.o.t..
$$
By considering the change of coordinates $u = \frac{1}{z}$ we obtain
$$
u_1 = u_0 + 1 + O(\frac{1}{u_0}),
$$
and we observe that if $r>0$ is chosen sufficiently small, the orbits of all initial values $z \in D(r,r) =\{|z-r|<r\}$ converge to the origin tangent to the positive real axis. In fact, after a slightly different change of coordinates one can obtain the simpler map
$$
u_1 = u_0 + 1.
$$
These coordinates on $D(r,r)$ are usually denoted by $u = \phi^i(z)$, and are referred to as the \emph{incoming Fatou coordinates}. The Fatou coordinates are invertible on a sufficiently small disk $D(r,r)$, and can be holomorphically extended to the whole parabolic basin by using the functional equation $\phi^i(f(z)) = \phi^i(z) + 1$.

By considering the inverse map $z_1 = z_0 + z_0^2 + h.o.t.$ we similarly obtain the outgoing Fatou coordinates $\phi^o$, defined on a small disk $D(-r,r)$. It is often convenient to use the inverse map of $\phi^o$, which we will denote by $\psi^o$. This inverse map can again be extended to all of $\mathbb C$ by using the functional equation $\psi^o(\zeta-1) = f(\psi^o(\zeta))$.

Now let $f$ be a rational function of degree at least $2$, and imagine that the parabolic basin does not contain a critical point. Then $\phi^i$ extends to a biholomorphic map from $\mathbb C$ to the parabolic basin. This gives a contradiction, as a parabolic basin must be a hyperbolic Riemann surface, i.e. its covering space is the unit disk, and therefore cannot be equivalent to $\mathbb C$. A similar argument can be given to deduce that any attracting basin must contain a critical point.

Let us return to the original maps $f_\lambda$. Recall that for fixed $d \ge 2$, we denote the region in parameter space $\mathbb C_\lambda$ for which $f_\lambda$ has an attracting fixed point by $U_d$. The set $U_d$ is an open and connected neighborhood of the origin. An immediate corollary of the above discussion is the following.

\begin{corollary}
For each $\lambda \in U_d$,  the orbit of the initial value
$$
z_0 = f^{\circ 2}_\lambda(\infty) = \lambda
$$
avoids the point $-1$.
\end{corollary}

In fact, it turns out that one can prove the following stronger statement.

\begin{lemma}\label{lemma:parabolic}
The region $U_d$ is a maximal open set of parameters $\lambda$ for which the orbit of $z_0$ avoids the critical point $-1$.
\end{lemma}

Observe that Lemma \ref{lemma:parabolic} directly implies Proposition~\ref{prop:regular trees}.

Note that the parameters $\lambda$ for which there is a parabolic fixed point form a dense subset of $\partial U_d$. Hence in order to obtain Lemma \ref{lemma:parabolic} it suffices to prove that for any parabolic parameter $\lambda_0 \in \partial U_d$ and any neighborhood $\mathcal{N}(\lambda_0)$, there exists a parameter $\lambda \in \mathcal{N}(\lambda_0)$ and an $N \in \mathbb N$ for which $f_{\lambda}^{\circ N}(z_0) = -1$. The fact that such $\lambda$ and $N$ exist is due to the following result regarding parabolic bifurcations.

\begin{theorem}\label{thm:bifurcation}
Let $f_\epsilon$ be a one-parameter family of rational functions that vary holomorphically with $\epsilon$. Assume that $f = f_0$ has a parabolic periodic cycle, and that this periodic cycle bifurcates for $\epsilon$ near $0$. Denote one of the corresponding parabolic basins by $\mathcal{B}_f$, let $z_0 \in \mathcal{B}_f$, and let $w \in \hat{\mathbb C} \setminus \mathcal{E}_f$. Then there exists a sequence of $\epsilon_j \rightarrow 0$ and $N_j \rightarrow \infty$ for which \change{$f_{\epsilon_j}^{\circ N_j}(z_0) = w$}.
\end{theorem}

Here $\mathcal{E}_f$ denotes the exceptional set, the largest finite completely invariant set, which by Montel's Theorem contains at most two points; see \cite{M06}. Since the set $\{-1,\infty\}$ containing the two critical points of the rational functions $f_\lambda: z \mapsto \frac{\lambda}{(1+z)^d}$  does not contain periodic orbits, it quickly follows that the exceptional set of these functions is empty. Lemma~\ref{lemma:parabolic} follows from Theorem~\ref{thm:bifurcation} by taking $w = -1$ and considering a sequence $(\lambda_j)$ that converges to a parabolic parameter $\lambda_0 \in \partial U_d$.

Perturbations of parabolic periodic points play a central role in complex dynamical systems, and have been studied extensively, see for example the classical works of Douady \cite{D94} and Lavaurs \cite{L89}. We will only give an indication of how to prove Theorem \ref{thm:bifurcation}, by discussing again the simplest model, $f(z) = z - z^2 + h.o.t.$, and $f_\epsilon(z) = f(z) + \epsilon^2$.  For $\epsilon \neq 0$, the unique parabolic fixed point $0 = f(0)$ splits up into two fixed points. For $\epsilon>0$ small these two fixed points are both close to the imaginary axis, forming a small ``gate'' for orbits to pass through.

For $\epsilon>0$ small enough, the orbit of an initial value $z_0 \in \mathcal{B}_f$, converging to $0$ under the original map $f$, will pass through the gate between these two fixed points, from the right to the left half plane. The time it takes to pass through the gate is roughly $\pi/\epsilon$. The following more precise statement was proved in \cite{L89}.

\begin{theorem}[Lavaurs, 89']
Let $\alpha \in \mathbb C$, and consider sequences $(\epsilon_j)$ of complex numbers satisfying $\epsilon_j \rightarrow 0$, and positive integers $(n_j)$ for which
$$
\frac{\pi}{\epsilon_j} - n_j \rightarrow \alpha.
$$
Then the maps $f_{\epsilon_j}^{\circ n_j}$ converge, uniformly on compact subsets of $\mathcal{B}_f$, to the map $\mathcal{L}_\alpha = \psi^o \circ T_\alpha \circ \phi^i$, where $T_\alpha$ denotes the translation $x \mapsto x+\alpha$.
\end{theorem}

Let $w \in \hat{\mathbb C} \setminus \mathcal{E}$, and let $\zeta_0 \in \mathbb C$ for which $\psi^o(\zeta_0) = w$.
Let $\alpha \in \mathbb C$ be given by
$$
\alpha = \zeta_0 - \phi^i(z_0)
$$
such that $\mathcal{L}_\alpha(z_0) = w$. Fix $\rho>0$ small, and for $\theta \in [0,2\pi]$ write
$$
\alpha(\theta) = \alpha + \rho e^{i\theta}
$$
and
$$
\epsilon_n(\theta) = \frac{\pi}{\alpha(\theta) + n}.
$$
It follows that
$$
f_{\epsilon_n(\theta)}^{\circ n}(z_0) \longrightarrow \mathcal{L}_{\alpha(\theta)} := \psi^o \circ T_{\alpha(\theta)} \circ \phi^i(z_0),
$$
uniformly over all $\theta \in [0,2\pi]$ as $n \rightarrow \infty$. Since the curve given by $\theta \mapsto \mathcal{L}_{\alpha(\theta)}(z_0)$ winds around $-1$, it follows that for $n$ sufficiently large there exists an $\alpha^\prime_n \in \mathcal{N}(\alpha,\rho)$ for which
$$
f_{\epsilon_n^\prime}^{\circ n}(z_0) = w
$$
is satisfied for
$$
\epsilon_n^\prime = \frac{\pi}{\alpha_n^\prime + n}.
$$
The general proof of Theorem \ref{thm:bifurcation} follows the same outline.

We end by proving that the literal statement of Theorem~\ref{thm:main} is false in the multivariate setting.

\begin{theorem}\label{thm:counter}
Let $\Delta \ge 3$ and let $D_\Delta$ be any neighborhood of the interval $[0, \lambda_\Delta)$. Then there exists a graph $G = (V,E)$ of maximum degree at most $\Delta$ and $\lambda \in D_\Delta^V$ such that $Z_G(\lambda) = 0$.
\end{theorem}

We will in fact use regular trees $G$ for which all vertices on a a given level will have the same values $\lambda_{v_i}$. In this setting we are dealing with a non-autonomous dynamical system given by the sequence
$$
x_k = \frac{\lambda_k}{(1+x_{k-1})^{\Delta-1}},
$$
with $x_0=0$ and where each $\lambda_k \in D_\Delta$. Hence Theorem~\ref{thm:counter} is implied by the following proposition.

\begin{prop}\label{prop:counter}
Given $\Delta$ and $D_\Delta$ as in Theorem~\ref{thm:counter}, there exist an integer $N \in \mathbb N$ and $\lambda_0, \ldots ,\lambda_N \in D_\Delta$ which give $x_N = -1$.
\end{prop}

The proof follows from the following lemma, which can be found in \cite{M06} and is a direct consequence of Montel's Theorem.

\begin{lemma}
Let $f$ be a rational function of degree at least $2$, let $x$ lie in the Julia set of $f$, and let $V$ be a neighborhood of $x$. Then
$$
\bigcup_{n \in \mathbb N} f^n(V) = \hat{\mathbb C} \setminus \mathcal{E}_f,
$$
where $\mathcal{E}_f$ is the exceptional set of $f$.
\end{lemma}

Let $\Delta \ge 3$ and $\lambda \neq 0$. As noted before in this appendix, the exceptional set of the function $f_{\Delta-1, \lambda}$ is empty. Thus, by compactness of the Riemann sphere, it follows that for any neighborhood $V$ of a point in the Julia set there exists an $N \in \mathbb N$ such that $f_{\Delta-1, \lambda}^N(V) = \hat{\mathbb C}$.

To prove Proposition~\ref{prop:counter}, let us denote the set of all possible values of points $x_N$ by $A$. Then $A$ contains $D_\Delta$, so in particular a neighborhood $V$ of the parabolic fixed point $x_\Delta$ of the function $f_{\Delta-1, \lambda_\Delta}$.

The parabolic fixed point $x_\Delta$ is contained in the Julia set of $f_{\Delta-1, \lambda_\Delta}$, thus it follows that there exists an $N \in \mathbb N$ for which $f_{\Delta-1, \lambda_\Delta}^N(V) = \hat{\mathbb C}$. But then $f_{\Delta-1, \lambda}^N(V) = \hat{\mathbb C}$ holds for $\lambda \in D$ sufficiently close to $\lambda_\Delta$, and thus $A = \hat{\mathbb C}$. But then $-1 \in A$, which completes the proof of Proposition~\ref{prop:counter}.

Note that in this construction the $\lambda_i$'s take on exactly two distinct values. On the lowest level of the tree they are very close to $x_\Delta$, and on all other levels they are very close to $\lambda_\Delta$. The thinner the set $D_\Delta$, the more levels the tree needs to have.

\section*{Acknowledgement}
We thank Heng Guo for pointing out an inaccuracy in an earlier version of this paper.
\change{We moreover thank Roland Roeder and Ivan Chio for helpful comments and spotting some typos. We also thank an anonymous referee for helpful comments.}

\end{appendix}

\end{document}